\documentclass[twoside,12pt,reqno]{amsart}

\usepackage{latexsym}

\newcommand{\qdn}{\hspace*{-1.5mm}}
\newcommand{\qqdn}{\hspace*{-2.5mm}}
\newcommand{\xqdn}{\hspace*{-5.0mm}}
\newcommand{\xxqdn}{\hspace*{-10mm}}

\newcommand{\fns}{\footnotesize}

\newcommand{\ffnk}[4]{\left[\qdn\ba{#1}#3\\[0.8mm]#4\ea{\!\Big|\:#2}\right]}

\newcommand{\binm}{\binom}

\newcommand{\be}{\begin{equation}}
\newcommand{\ee}{\end{equation}}
\newcommand{\ba}{\begin{array}}
\newcommand{\ea}{\end{array}}
\newcommand{\bmn}{\begin{eqnarray}}
\newcommand{\emn}{\end{eqnarray}}
\newcommand{\bnm}{\begin{eqnarray*}}
\newcommand{\enm}{\end{eqnarray*}}
\newcommand{\bln}{\begin{subequations}}
\newcommand{\eln}{\end{subequations}}

\newtheorem{thm}{Theorem}
\newtheorem{lemm}[thm]{Lemma}

\newtheorem{entry}{Entry}

\newcommand{\bbtm}[4]{\bibitem{kn:#1}{#2,}~{#3,}~{#4.}}
\newcommand{\cito}[1]{\cite{kn:#1}}

\begin{document} 
{\fns
\title{Pell's equation and series expansions \\for irrational numbers}
\author{$^{1,2}$Chuanan Wei}
\dedicatory{
  $^1$Department of Mathematics\\
  Shanghai Normal University, Shanghai 200234, China\\
  $^2$Department of Information Technology\\
  Hainan Medical College, Haikou 571199, China}
\thanks{Email address: weichuanan78@163.com}

\address{ }


\begin{abstract}
Solutions of Pell's equation and hypergeometric series identities
are used to study series expansions for $\sqrt{p}$ where $p$ are
arbitrary prime numbers. Numerous fast convergent series expansions
for this family of irrational numbers are established.
\end{abstract}

\maketitle\thispagestyle{empty}
\markboth{Chuanan Wei}
         {Series expansions for irrational numbers}

\section{Introduction}
Pell's equation(also called the Pell-Fermat equation) is any
Diophantine equation of the form
 \bmn\label{pell}
  x^2-py^2=1,
 \emn
   where $p$ is a given positive nonsquare integer and integer solutions are sought
for $x$ and $y$. Joseph Louis Lagrange proved that Pell's equation
has infinitely many distinct integer solutions. Furthermore, there
holds the following relation.
\begin{lemm} \label{lemm-a}
Let $s$ be a positive integer. If $(x_1,y_1)$ is the integer
 solution to \eqref{pell}, then $(x_s,y_s)$ with
\bnm\qquad
\begin{cases}
x_s=\sum_{k=0}^{\lfloor\frac{s}{2}\rfloor}\binm{s}{2k}p^ky_1^{2k}x_1^{s-2k},\\
y_s=\sum_{k=0}^{\lfloor\frac{s-1}{2}\rfloor}\binm{s}{1+2k}p^ky_1^{1+2k}x_1^{s-1-2k}
\end{cases}
\enm
  is also the integer
 solution of \eqref{pell}.
\end{lemm}

\begin{proof}
Because that $(x_1,y_1)$ is the integer
 solution to \eqref{pell}, we obtain
$$(x_1+\sqrt{p}\,y_1)^s(x_1-\sqrt{p}\,y_1)^s=1.$$
In terms of the binomial theorem
 \bnm
(u+v)^t=\sum_{k=0}^t\binm{t}{k}u^kv^{t-k},
 \enm
we get the following two expansions
 \bnm
 (x_1+\sqrt{p}\,y_1)^s
  &&\xqdn=\,\sum_{k=0}^s\binm{s}{k}(\sqrt{p}y_1)^kx_1^{s-k}\\
  &&\xqdn=\,\sum_{k=0}^{\lfloor\frac{s}{2}\rfloor}\binm{s}{2k}p^ky_1^{2k}x_1^{s-2k}\\
  &&\xqdn+\,\sqrt{p}\sum_{k=0}^{\lfloor\frac{s-1}{2}\rfloor}\binm{s}{1+2k}p^ky_1^{1+2k}x_1^{s-1-2k},\\
  (x_1-\sqrt{p}\,y_1)^s
  &&\xqdn=\,\sum_{k=0}^s\binm{s}{k}(-\sqrt{p}y_1)^kx_1^{s-k}\\
  &&\xqdn=\,\sum_{k=0}^{\lfloor\frac{s}{2}\rfloor}\binm{s}{2k}p^ky_1^{2k}x_1^{s-2k}\\
  &&\xqdn-\,\sqrt{p}\sum_{k=0}^{\lfloor\frac{s-1}{2}\rfloor}\binm{s}{1+2k}p^ky_1^{1+2k}x_1^{s-1-2k}.
  \enm
So we gain
 \bnm
\Bigg\{\sum_{k=0}^{\lfloor\frac{s}{2}\rfloor}\binm{s}{2k}p^ky_1^{2k}x_1^{s-2k}\Bigg\}^2
-p\Bigg\{\sum_{k=0}^{\lfloor\frac{s-1}{2}\rfloor}\binm{s}{1+2k}p^ky_1^{1+2k}x_1^{s-1-2k}\Bigg\}^2=1.
 \enm
This completes the proof of Lemma \ref{lemm-a}.
\end{proof}

The circumference ratio $\pi=3.1415926535\cdots$ is one of the most
important irrational numbers. For centuries, the study of
$\pi$-formulas attracted many persons. Recently, Chu \cito{chu-a},
Liu \cite{kn:liu-a,kn:liu-b} and Wei and Gong \cito{wei-a} gave many
$\pi$-formulas in terms of the hypergeometric method. Different
methods and results can be seen in the papers
\cite{kn:chu-b,kn:guillera-a,kn:guillera-b,kn:guillera-d,kn:guillera-e,kn:wei-b}.
 For historical
notes and introductory informations on this kind of series, the
readers may refer to four surveys
\cite{kn:bailey-a,kn:baruah,kn:guillera-c,kn:zudilin}.

It is well known that $\sqrt{p}$ are irrational numbers when $p$ are
arbitrary prime numbers. Several ones of this family of irrational
numbers are closely related to $\pi$. For example, we have the
following relations:
 \bmn
&&\sin\frac{\pi}{4}=\frac{\sqrt{2}}{2}, \label{equation-a}\\
&&\sin\frac{\pi}{3}=\frac{\sqrt{3}}{2},  \label{equation-b}\\
&&\sin\frac{\pi}{10}=\frac{\sqrt{5}-1}{4}, \label{equation-c}\\
&&\frac{\sqrt{7}+1}{6\sqrt{3}}\pi
=\sum_{k=0}^{\infty}\frac{(-1)^{\binm{k}{2}}}{2k+1}\bigg(\frac{3}{4+\sqrt{7}}\bigg)^k, \label{equation-d}\\
&&\frac{\pi}{3\sqrt{3}}=\sum_{k=0}^{\infty}\frac{1}{(649+180\sqrt{13})^k}
\bigg\{\frac{\sqrt{13}-3}{6k+1}-\frac{109\sqrt{13}-393}{6k+5}\bigg\},\label{equation-e}
 \emn
where \eqref{equation-a}-\eqref{equation-c} are proverbial and
\eqref{equation-d}, \eqref{equation-e} can be seen in Wei
\cito{wei-b}.

 For a complex number $x$, define the shifted factorial to be
 \bnm
 (x)_0=1\quad\text{and}\quad(x)_n=x(x+1)\cdots(x+n-1)\quad\text{when}\quad n\in \mathbb{N}.
 \enm
 Following Bailey~\cito{bailey}, define the hypergeometric series by
\[_{1+r}F_s\ffnk{cccc}{z}{a_0,&a_1,&\cdots,&a_r}{&b_1,&\cdots,&b_s}
 \:=\:\sum_{k=0}^\infty\frac{(a_0)_k(a_1)_k\cdots(a_r)_k}{(1)_k(b_1)_k\cdots(b_s)_k}z^k,\]
where $\{a_{i}\}_{i\geq0}$ and $\{b_{j}\}_{j\geq1}$ are complex
parameters such that no zero factors appear in the denominators of
the summand on the right hand side. Then three hypergeometric series
identities can be stated as follows:
 \bmn
 &&\label{expansion-a}
 \xxqdn{_1F_0}\ffnk{cccc}{x}{a}{-}=\frac{1}{(1-x)^a}\quad\text{with}\quad
 |x|<1,\\
 &&\label{expansion-b}
 \xxqdn{_2F_1}\ffnk{cccc}{\frac{4x}{(1+x)^2}}{\frac{1}{2}a,\frac{1}{2}+\frac{1}{2}a}{1+a}
  =(1+x)^a\quad\text{with}\quad
 \bigg|\frac{4x}{(1+x)^2}\bigg|<1,\\
 &&\label{expansion-c}
 \xxqdn{_3F_2}\ffnk{cccc}{\frac{27x}{4(1+x)^3}}{\frac{1}{3}a,\frac{1}{3}+\frac{1}{3}a,\frac{2}{3}+\frac{1}{3}a}
 {\frac{1}{2}+\frac{1}{2}a,1+\frac{1}{2}a}
  =(1+x)^a\quad\text{with}\quad
 \bigg|\frac{27x}{4(1+x)^3}\bigg|<1,
 \emn
where \eqref{expansion-a} is a well-known identity and
\eqref{expansion-b}, \eqref{expansion-c} can be found in Gessel and
Stanton \cito{gessel}.

 Inspired by the works just mentioned, we shall
systematically explore series expansions for $\sqrt{p}$ with $p$
being prime numbers by means of Lemma \ref{lemm-a} and
\eqref{expansion-a}-\eqref{expansion-c}. The structure of the paper
is arranged as follows. We shall establish six theorems in Section
2. Then they and \emph{Mathematica} program are utilized to produce
concrete series expansions for $\sqrt{p}$ in Sections 3-8.

\section{Six Theorems}

\begin{thm}\label{thm-a}
Let $p$ be a positive nonsquare integer and $m,n$ be both positive
integers satisfying $n^2-pm^2=1$. Then
 \bnm
\sqrt{p}=\frac{mp}{n}\sum_{k=0}^{\infty}\frac{(1/2)_k}{k!}\bigg(\frac{1}{pm^2+1}\bigg)^k.
 \enm
\end{thm}

\begin{proof}
The case $a=1/2$ of \eqref{expansion-a} reads as
  \bmn \label{expansion-d}
\frac{1}{\sqrt{1-x}}=\sum_{k=0}^{\infty}\frac{(1/2)_k}{k!}x^k\quad\text{with}\quad
 |x|<1.
 \emn
Setting $x=1/(pm^2+1)$ in \eqref{expansion-d}, we achieve
 \bnm
\sqrt{\frac{pm^2+1}{pm^2}}=\sum_{k=0}^{\infty}\frac{(1/2)_k}{k!}\bigg(\frac{1}{pm^2+1}\bigg)^k.
 \enm
When $pm^2+1=n^2$, the last equation becomes
 \bnm
\frac{n}{m}\frac{1}{\sqrt{p}}=\sum_{k=0}^{\infty}\frac{(1/2)_k}{k!}\bigg(\frac{1}{pm^2+1}\bigg)^k.
 \enm
Multiplying both sides by $mp/n$,
 we attain Theorem \ref{thm-a} to finish the proof.
\end{proof}

\begin{thm}\label{thm-b}
Let $p$ be a positive nonsquare integer and $m,n$ be both positive
integers provided that $n^2-pm^2=1$. Then
 \bnm
\sqrt{p}=\frac{n}{m}\sum_{k=0}^{\infty}\frac{(1/2)_k}{k!}\bigg(-\frac{1}{pm^2}\bigg)^k.
 \enm
\end{thm}

\begin{proof}
Taking $x=-1/pm^2$ in \eqref{expansion-d}, we obtain
 \bnm
\sqrt{\frac{pm^2}{pm^2+1}}=\sum_{k=0}^{\infty}\frac{(1/2)_k}{k!}\bigg(-\frac{1}{pm^2}\bigg)^k.
 \enm
When $pm^2+1=n^2$, the last equation creates
 \bnm
\frac{m}{n}\sqrt{p}=\sum_{k=0}^{\infty}\frac{(1/2)_k}{k!}\bigg(-\frac{1}{pm^2}\bigg)^k.
 \enm
Multiplying both sides by $n/m$,
 we get Theorem \ref{thm-b} to complete the proof.
\end{proof}

\begin{thm}\label{thm-c}
Let $p$ be a positive nonsquare integer and $m,n$ be both positive
integers satisfying $n^2-pm^2=1$. Then
 \bnm
\sqrt{p}=\frac{mp}{n}\sum_{k=0}^{\infty}\frac{(1/4)_k(3/4)_k}{k!(3/2)_k}\bigg\{\frac{4pm^2}{(pm^2+1)^2}\bigg\}^{k}.
 \enm
\end{thm}

\begin{proof}
The case $a=1/2$ of \eqref{expansion-b} gives
  \bmn \label{expansion-e}
\sqrt{1+x}=\sum_{k=0}^{\infty}\frac{(1/4)_k(3/4)_k}{k!(3/2)_k}\bigg\{\frac{4x}{(1+x)^2}\bigg\}^k
\quad\text{with}\quad
 \bigg|\frac{4x}{(1+x)^2}\bigg|<1.
 \emn
Fixing $x=1/pm^2$ in \eqref{expansion-e}, we gain
 \bnm
\sqrt{\frac{pm^2+1}{pm^2}}=\sum_{k=0}^{\infty}\frac{(1/4)_k(3/4)_k}{k!(3/2)_k}\bigg\{\frac{4pm^2}{(pm^2+1)^2}\bigg\}^{k}.
 \enm
When $pm^2+1=n^2$, the last equation produces
 \bnm
\frac{n}{m}\frac{1}{\sqrt{p}}=\sum_{k=0}^{\infty}\frac{(1/4)_k(3/4)_k}{k!(3/2)_k}\bigg\{\frac{4pm^2}{(pm^2+1)^2}\bigg\}^{k}.
 \enm
Multiplying both sides by $mp/n$,
 we achieve Theorem \ref{thm-c} to finish the proof.
\end{proof}

\begin{thm}\label{thm-d}
Let $p$ be a positive nonsquare integer and $m,n$ be both positive
integers provided that $n^2-pm^2=1$ and $p^2m^4>4pm^2+4$. Then
 \bnm
\sqrt{p}=\frac{n}{m}\sum_{k=0}^{\infty}\frac{(1/4)_k(3/4)_k}{k!(3/2)_k}\bigg\{-\frac{4(pm^2+1)}{p^2m^4}\bigg\}^{k}.
 \enm
\end{thm}

\begin{proof}
Setting $x=-1/(pm^2+1)$ in \eqref{expansion-e}, we attain
 \bnm
\sqrt{\frac{pm^2}{pm^2+1}}=\sum_{k=0}^{\infty}\frac{(1/4)_k(3/4)_k}{k!(3/2)_k}\bigg\{-\frac{4(pm^2+1)}{p^2m^4}\bigg\}^{k}.
 \enm
When $pm^2+1=n^2$, the last equation becomes
 \bnm
\frac{m}{n}\sqrt{p}=\sum_{k=0}^{\infty}\frac{(1/4)_k(3/4)_k}{k!(3/2)_k}\bigg\{-\frac{4(pm^2+1)}{p^2m^4}\bigg\}^{k}.
 \enm
Multiplying both sides by $n/m$,
 we obtain Theorem \ref{thm-d} to complete the proof.
\end{proof}

\begin{thm}\label{thm-e}
Let $p$ be a positive nonsquare integer and $m,n$ be both positive
integers satisfying $n^2-pm^2=1$ and $4(pm^2+1)^3>27p^2m^4$. Then
 \bnm
\sqrt{p}=\frac{mp}{n}\sum_{k=0}^{\infty}\frac{(1/2)_k(1/6)_k(5/6)_k}{k!(3/4)_k(5/4)_k}
\bigg\{\frac{27p^2m^4}{4(pm^2+1)^3}\bigg\}^{k}.
 \enm
\end{thm}

\begin{proof}
The case $a=1/2$ of \eqref{expansion-c} offers
  \bmn \label{expansion-f}
\sqrt{1+x}=\sum_{k=0}^{\infty}\frac{(1/2)_k(1/6)_k(5/6)_k}{k!(3/4)_k(5/4)_k}\bigg\{\frac{27x}{4(1+x)^3}\bigg\}^k
\quad\text{with}\quad
 \bigg|\frac{27x}{4(1+x)^3}\bigg|<1.
 \emn
Taking $x=1/pm^2$ in \eqref{expansion-f}, we get
 \bnm
\sqrt{\frac{pm^2+1}{pm^2}}=\sum_{k=0}^{\infty}\frac{(1/2)_k(1/6)_k(5/6)_k}{k!(3/4)_k(5/4)_k}
\bigg\{\frac{27p^2m^4}{4(pm^2+1)^3}\bigg\}^{k}.
 \enm
When $pm^2+1=n^2$, the last equation creates
 \bnm
 \frac{n}{m}\frac{1}{\sqrt{p}}=\sum_{k=0}^{\infty}\frac{(1/2)_k(1/6)_k(5/6)_k}{k!(3/4)_k(5/4)_k}
\bigg\{\frac{27p^2m^4}{4(pm^2+1)^3}\bigg\}^{k}.
 \enm
 Multiplying both sides by $mp/n$,
 we gain Theorem \ref{thm-e} to finish the proof.
\end{proof}

\begin{thm}\label{thm-f}
Let $p$ be a positive nonsquare integer and $m,n$ be both positive
integers provided that $n^2-pm^2=1$ and $4p^3m^6>27(pm^2+1)^2$. Then
 \bnm
\sqrt{p}=\frac{n}{m}\sum_{k=0}^{\infty}\frac{(1/2)_k(1/6)_k(5/6)_k}{k!(3/4)_k(5/4)_k}
\bigg\{-\frac{27(pm^2+1)^2}{4p^3m^6}\bigg\}^{k}.
 \enm
\end{thm}

\begin{proof}
Fixing $x=-1/(pm^2+1)$ in \eqref{expansion-f}, we achieve
 \bnm
\sqrt{\frac{pm^2}{pm^2+1}}=\sum_{k=0}^{\infty}\frac{(1/2)_k(1/6)_k(5/6)_k}{k!(3/4)_k(5/4)_k}
\bigg\{-\frac{27(pm^2+1)^2}{4p^3m^6}\bigg\}^{k}.
 \enm
When $pm^2+1=n^2$, the last equation produces
 \bnm
\frac{m}{n}\sqrt{p}=\sum_{k=0}^{\infty}\frac{(1/2)_k(1/6)_k(5/6)_k}{k!(3/4)_k(5/4)_k}
\bigg\{-\frac{27(pm^2+1)^2}{4p^3m^6}\bigg\}^{k}.
 \enm
Multiplying both sides by $n/m$,
 we attain Theorem \ref{thm-f} to complete the proof.
\end{proof}
\section{Series expansions for $\sqrt{2}$}
Setting $p=2$ in \eqref{pell}, we obtain
 \bmn\label{pell-a}
  x^2-2y^2=1.
 \emn
It is easy to know that $x_1=3, y_1=2$ is the solution to
\eqref{pell-a}. So
 \bnm\qquad
\begin{cases}
x_s=\sum_{k=0}^{\lfloor\frac{s}{2}\rfloor}\binm{s}{2k}2^{3k}3^{s-2k},\\
y_s=\sum_{k=0}^{\lfloor\frac{s-1}{2}\rfloor}\binm{s}{1+2k}2^{1+3k}3^{s-1-2k}
\end{cases}
\enm
 is also the solution of \eqref{pell-a} thanks to Lemma
 \ref{lemm-a}. Now we choose $x_4=577, y_4=408$ and $x_7=114243, y_7=80782$
 to give 12 series expansions for $\sqrt{2}$.

Substituting p=2, $n=x_4=577$ and $m=y_4=408$ into Theorems
\ref{thm-a}-\ref{thm-f}, we get the following six series expansions
for $\sqrt{2}$\,:
 \bnm
&&\xxqdn\xxqdn\xxqdn\xqdn\qdn\sqrt{2}=\frac{816}{577}\sum_{k=0}^{\infty}\frac{(1/2)_k}{k!}\bigg(\frac{1}{332929}\bigg)^k,\\
&&\xxqdn\xxqdn\xxqdn\xqdn\qdn\sqrt{2}=\frac{577}{408}\sum_{k=0}^{\infty}\frac{(1/2)_k}{k!}\bigg(-\frac{1}{332928}\bigg)^k,\\
&&\xxqdn\xxqdn\xxqdn\xqdn\qdn\sqrt{2}=\frac{816}{577}\sum_{k=0}^{\infty}\frac{(1/4)_k(3/4)_k}{k!(3/2)_k}\bigg(\frac{1331712}{110841719041}\bigg)^{k},\\
&&\xxqdn\xxqdn\xxqdn\xqdn\qdn\sqrt{2}=\frac{577}{408}\sum_{k=0}^{\infty}\frac{(1/4)_k(3/4)_k}{k!(3/2)_k}\bigg(-\frac{332929}{27710263296}\bigg)^{k},\\
&&\xxqdn\xxqdn\xxqdn\xqdn\qdn\sqrt{2}=\frac{816}{577}\sum_{k=0}^{\infty}\frac{(1/2)_k(1/6)_k(5/6)_k}{k!(3/4)_k(5/4)_k}
\bigg(\frac{748177108992}{36902422678601089}\bigg)^{k},\\
&&\xxqdn\xxqdn\xxqdn\xqdn\qdn\sqrt{2}=\frac{577}{408}\sum_{k=0}^{\infty}\frac{(1/2)_k(1/6)_k(5/6)_k}{k!(3/4)_k(5/4)_k}
\bigg(-\frac{110841719041}{5466976319176704}\bigg)^{k}.
 \enm
 Substituting p=2, $n=x_7=114243$ and $m=y_7=80782$ into Theorems
\ref{thm-a}-\ref{thm-f}, we gain the following six series expansions
for $\sqrt{2}$\,:
 \bnm
&&\xqdn\sqrt{2}=\frac{161564}{114243}\sum_{k=0}^{\infty}\frac{(1/2)_k}{k!}\bigg(\frac{1}{13051463049}\bigg)^k,\\
&&\xqdn\sqrt{2}=\frac{114243}{80782}\sum_{k=0}^{\infty}\frac{(1/2)_k}{k!}\bigg(-\frac{1}{13051463048}\bigg)^k,\\
&&\xqdn\sqrt{2}=\frac{161564}{114243}\sum_{k=0}^{\infty}\frac{(1/4)_k(3/4)_k}{k!(3/2)_k}
\bigg(\frac{52205852192}{170340687719412376401}\bigg)^{k},\\
&&\xqdn\sqrt{2}=\frac{114243}{80782}\sum_{k=0}^{\infty}\frac{(1/4)_k(3/4)_k}{k!(3/2)_k}
\bigg(-\frac{13051463049}{42585171923327362576}\bigg)^{k},\\
&&\xqdn\sqrt{2}=\frac{161564}{114243}\sum_{k=0}^{\infty}\frac{(1/2)_k(1/6)_k(5/6)_k}{k!(3/4)_k(5/4)_k}
 \bigg(\frac{42585171923327362576}{82340562648561433725590040987}\bigg)^{k},\\
&&\xqdn\sqrt{2}=\frac{114243}{80782}\sum_{k=0}^{\infty}\frac{(1/2)_k(1/6)_k(5/6)_k}{k!(3/4)_k(5/4)_k}
\bigg(-\frac{4599198568424134162827}{8892780764000546589887393466368}\bigg)^{k}.
 \enm
Numerous different series expansions for $\sqrt{2}$ can be derived
in the same way. Due to limit of space, the corresponding results
will not be displayed in the paper. The discuss is also adapt to
series expansions for $\sqrt{p}$ with $p>2$.
\section{Series expansions for $\sqrt{3}$}
Taking $p=3$ in \eqref{pell}, we have
 \bmn\label{pell-b}
  x^2-3y^2=1.
 \emn
It is not difficult to see that $x_1=2, y_1=1$ is the solution to
\eqref{pell-b}. Thus
 \bnm\qquad
\begin{cases}
x_s=\sum_{k=0}^{\lfloor\frac{s}{2}\rfloor}\binm{s}{2k}3^k2^{s-2k},\\
y_s=\sum_{k=0}^{\lfloor\frac{s-1}{2}\rfloor}\binm{s}{1+2k}3^k2^{s-1-2k}
\end{cases}
\enm
 is also the solution of \eqref{pell-b} according to Lemma
 \ref{lemm-a}. Now we select $x_5=362, y_5=209$ and $x_9=70226, y_9=40545$
 to create 12 series expansions for $\sqrt{3}$.

Substituting p=3, $n=x_5=362$ and $m=y_5=209$ into Theorems
\ref{thm-a}-\ref{thm-f}, we achieve the following six series
expansions for $\sqrt{3}$\,:
 \bnm
&&\xxqdn\xxqdn\xxqdn\xqdn\sqrt{3}=\frac{627}{362}\sum_{k=0}^{\infty}\frac{(1/2)_k}{k!}\bigg(\frac{1}{131044}\bigg)^k,\\
&&\xxqdn\xxqdn\xxqdn\xqdn\sqrt{3}=\frac{362}{209}\sum_{k=0}^{\infty}\frac{(1/2)_k}{k!}\bigg(-\frac{1}{131043}\bigg)^k,\\
&&\xxqdn\xxqdn\xxqdn\xqdn\sqrt{3}=\frac{627}{362}\sum_{k=0}^{\infty}
\frac{(1/4)_k(3/4)_k}{k!(3/2)_k}\bigg(\frac{131043}{4293132484}\bigg)^{k},\\
&&\xxqdn\xxqdn\xxqdn\xqdn\sqrt{3}=\frac{362}{209}\sum_{k=0}^{\infty}
\frac{(1/4)_k(3/4)_k}{k!(3/2)_k}\bigg(-\frac{524176}{17172267849}\bigg)^{k},\\
&&\xxqdn\xxqdn\xxqdn\xqdn\sqrt{3}=\frac{627}{362}\sum_{k=0}^{\infty}
\frac{(1/2)_k(1/6)_k(5/6)_k}{k!(3/4)_k(5/4)_k}
\bigg(\frac{463651231923}{9001428051732736}\bigg)^{k},\\
&&\xxqdn\xxqdn\xxqdn\xqdn\sqrt{3}=\frac{362}{209}\sum_{k=0}^{\infty}
\frac{(1/2)_k(1/6)_k(5/6)_k}{k!(3/4)_k(5/4)_k}
\bigg(-\frac{4293132484}{83344647990241}\bigg)^{k}.
 \enm

Substituting p=3, $n=x_9=70226$ and $m=y_9=40545$ into Theorems
\ref{thm-a}-\ref{thm-f}, we attain the following six series
expansions for $\sqrt{3}$\,:
 \bnm
&&\xqdn\sqrt{3}=\frac{121635}{70226}\sum_{k=0}^{\infty}\frac{(1/2)_k}{k!}\bigg(\frac{1}{4931691076}\bigg)^k,\\
&&\xqdn\sqrt{3}=\frac{70226}{40545}\sum_{k=0}^{\infty}\frac{(1/2)_k}{k!}\bigg(-\frac{1}{4931691075}\bigg)^k,\\
&&\xqdn\sqrt{3}=\frac{121635}{70226}\sum_{k=0}^{\infty}
\frac{(1/4)_k(3/4)_k}{k!(3/2)_k}\bigg(\frac{4931691075}{6080394217274509444}\bigg)^{k},\\
&&\xqdn\sqrt{3}=\frac{70226}{40545}\sum_{k=0}^{\infty}
\frac{(1/4)_k(3/4)_k}{k!(3/2)_k}\bigg(-\frac{19726764304}{24321576859234655625}\bigg)^{k},\\
&&\xqdn\sqrt{3}=\frac{121635}{70226}\sum_{k=0}^{\infty}
\frac{(1/2)_k(1/6)_k(5/6)_k}{k!(3/4)_k(5/4)_k}
\bigg(\frac{656682575199335701875}{479786014398315252276040347904}\bigg)^{k},\\
&&\xqdn\sqrt{3}=\frac{70226}{40545}\sum_{k=0}^{\infty}
\frac{(1/2)_k(1/6)_k(5/6)_k}{k!(3/4)_k(5/4)_k}
\bigg(-\frac{6080394217274509444}{4442463093578299350981890625}\bigg)^{k}.
 \enm

\section{Series expansions for $\sqrt{5}$}
Fixing $p=5$ in \eqref{pell}, we obtain
 \bmn\label{pell-c}
  x^2-5y^2=1.
 \emn
It is ordinary to find that $x_1=9, y_1=4$ is the solution to
\eqref{pell-c}. Therefore
 \bnm\qquad
\begin{cases}
x_s=\sum_{k=0}^{\lfloor\frac{s}{2}\rfloor}\binm{s}{2k}5^{k}4^{2k}9^{s-2k},\\
y_s=\sum_{k=0}^{\lfloor\frac{s-1}{2}\rfloor}\binm{s}{1+2k}5^{k}4^{1+2k}9^{s-1-2k}
\end{cases}
\enm
 is also the solution of \eqref{pell-c} in accordance with Lemma
 \ref{lemm-a}. Now we choose $x_3=2889, y_3=1292$ and $x_4=51841, y_4=23184$
 to produce 12 series expansions for $\sqrt{5}$.

Substituting p=5, $n=x_3=2889$ and $m=y_3=1292$ into Theorems
\ref{thm-a}-\ref{thm-f}, we get the following six series expansions
for $\sqrt{5}$\,:
 \bnm
&&\xxqdn\xxqdn\qqdn\sqrt{5}=\frac{6460}{2889}\sum_{k=0}^{\infty}\frac{(1/2)_k}{k!}\bigg(\frac{1}{8346321}\bigg)^k,\\
&&\xxqdn\xxqdn\qqdn\sqrt{5}=\frac{2889}{1292}\sum_{k=0}^{\infty}\frac{(1/2)_k}{k!}\bigg(-\frac{1}{8346320}\bigg)^k,\\
&&\xxqdn\xxqdn\qqdn\sqrt{5}=\frac{6460}{2889}\sum_{k=0}^{\infty}
\frac{(1/4)_k(3/4)_k}{k!(3/2)_k}\bigg(\frac{33385280}{69661074235041}\bigg)^{k},\\
&&\xxqdn\xxqdn\qqdn\sqrt{5}=\frac{2889}{1292}\sum_{k=0}^{\infty}
\frac{(1/4)_k(3/4)_k}{k!(3/2)_k}\bigg(-\frac{8346321}{17415264385600}\bigg)^{k},\\
&&\xxqdn\xxqdn\qqdn\sqrt{5}=\frac{6460}{2889}\sum_{k=0}^{\infty}
\frac{(1/2)_k(1/6)_k(5/6)_k}{k!(3/4)_k(5/4)_k}
\bigg(\frac{17415264385600}{21533840250758579043}\bigg)^{k},\\
&&\xxqdn\xxqdn\qqdn\sqrt{5}=\frac{2889}{1292}\sum_{k=0}^{\infty}
\frac{(1/2)_k(1/6)_k(5/6)_k}{k!(3/4)_k(5/4)_k}
\bigg(-\frac{1880849004346107}{2325653911149135872000}\bigg)^{k}.
 \enm

Substituting p=5, $n=x_4=51841$ and $m=y_4=23184$ into Theorems
\ref{thm-a}-\ref{thm-f}, we gain the following six series expansions
for $\sqrt{5}$\,:
 \bnm
&&\xxqdn\sqrt{5}=\frac{115920}{51841}\sum_{k=0}^{\infty}\frac{(1/2)_k}{k!}\bigg(\frac{1}{2687489281}\bigg)^k,\\
&&\xxqdn\sqrt{5}=\frac{51841}{23184}\sum_{k=0}^{\infty}\frac{(1/2)_k}{k!}\bigg(-\frac{1}{2687489280}\bigg)^k,\\
&&\xxqdn\sqrt{5}=\frac{115920}{51841}\sum_{k=0}^{\infty}
\frac{(1/4)_k(3/4)_k}{k!(3/2)_k}\bigg(\frac{10749957120}{7222598635489896961}\bigg)^{k},\\
&&\xxqdn\sqrt{5}=\frac{51841}{23184}\sum_{k=0}^{\infty}
\frac{(1/4)_k(3/4)_k}{k!(3/2)_k}\bigg(-\frac{2687489281}{1805649657528729600}\bigg)^{k},\\
&&\xxqdn\sqrt{5}=\frac{115920}{51841}\sum_{k=0}^{\infty}
\frac{(1/2)_k(1/6)_k(5/6)_k}{k!(3/4)_k(5/4)_k}
\bigg(\frac{48752540753275699200}{19410656413844324266481975041}\bigg)^{k},\\
&&\xxqdn\sqrt{5}=\frac{51841}{23184}\sum_{k=0}^{\infty}
\frac{(1/2)_k(1/6)_k(5/6)_k}{k!(3/4)_k(5/4)_k}
\bigg(-\frac{7222598635489896961}{2875652798840967165640704000}\bigg)^{k}.
 \enm

\section{Series expansions for $\sqrt{7}$}
Setting $p=7$ in \eqref{pell}, we have
 \bmn\label{pell-d}
  x^2-7y^2=1.
 \emn
It is easy to know that $x_1=8, y_1=3$ is the solution to
\eqref{pell-d}. So
 \bnm\qquad
\begin{cases}
x_s=\sum_{k=0}^{\lfloor\frac{s}{2}\rfloor}\binm{s}{2k}7^{k}3^{2k}8^{s-2k},\\
y_s=\sum_{k=0}^{\lfloor\frac{s-1}{2}\rfloor}\binm{s}{1+2k}7^{k}3^{1+2k}8^{s-1-2k}
\end{cases}
\enm
 is also the solution of \eqref{pell-d} thanks to Lemma
 \ref{lemm-a}. Now we select $x_3=2024, y_3=765$ and $x_4=32257, y_4=12192$
 to give 12 series expansions for $\sqrt{7}$.

Substituting p=7, $n=x_3=2024$ and $m=y_3=765$ into Theorems
\ref{thm-a}-\ref{thm-f}, we achieve the following six series
expansions for $\sqrt{7}$\,:
 \bnm
&&\xxqdn\xxqdn\xqdn\sqrt{7}=\frac{5355}{2024}\sum_{k=0}^{\infty}\frac{(1/2)_k}{k!}\bigg(\frac{1}{4096576}\bigg)^k,\\
&&\xxqdn\xxqdn\xqdn\sqrt{7}=\frac{2024}{765}\sum_{k=0}^{\infty}\frac{(1/2)_k}{k!}\bigg(-\frac{1}{4096575}\bigg)^k,\\
&&\xxqdn\xxqdn\xqdn\sqrt{7}=\frac{5355}{2024}\sum_{k=0}^{\infty}
\frac{(1/4)_k(3/4)_k}{k!(3/2)_k}\bigg(\frac{4096575}{4195483730944}\bigg)^{k},\\
&&\xxqdn\xxqdn\xqdn\sqrt{7}=\frac{2024}{765}\sum_{k=0}^{\infty}
\frac{(1/4)_k(3/4)_k}{k!(3/2)_k}\bigg(-\frac{16386304}{16781926730625}\bigg)^{k},\\
&&\xxqdn\xxqdn\xqdn\sqrt{7}=\frac{5355}{2024}\sum_{k=0}^{\infty}
\frac{(1/2)_k(1/6)_k(5/6)_k}{k!(3/4)_k(5/4)_k}
\bigg(\frac{453112021726875}{274993887369210363904}\bigg)^{k},\\
&&\xxqdn\xxqdn\xqdn\sqrt{7}=\frac{2024}{765}\sum_{k=0}^{\infty}
\frac{(1/2)_k(1/6)_k(5/6)_k}{k!(3/4)_k(5/4)_k}
\bigg(-\frac{4195483730944}{2546237833204078125}\bigg)^{k}.
 \enm

Substituting p=7, $n=x_4=32257$ and $m=y_4=12192$ into Theorems
\ref{thm-a}-\ref{thm-f}, we attain the following six series
expansions for $\sqrt{7}$\,:
 \bnm
&&\xxqdn\sqrt{7}=\frac{85344}{32257}\sum_{k=0}^{\infty}\frac{(1/2)_k}{k!}\bigg(\frac{1}{1040514049}\bigg)^k,\\
&&\xxqdn\sqrt{7}=\frac{32257}{12192}\sum_{k=0}^{\infty}\frac{(1/2)_k}{k!}\bigg(-\frac{1}{1040514048}\bigg)^k,\\
&&\xxqdn\sqrt{7}=\frac{85344}{32257}\sum_{k=0}^{\infty}
\frac{(1/4)_k(3/4)_k}{k!(3/2)_k}\bigg(\frac{4162056192}{1082669486166374401}\bigg)^{k},\\
&&\xxqdn\sqrt{7}=\frac{32257}{12192}\sum_{k=0}^{\infty}
\frac{(1/4)_k(3/4)_k}{k!(3/2)_k}\bigg(-\frac{1040514049}{270667371021336576}\bigg)^{k},\\
&&\xxqdn\sqrt{7}=\frac{85344}{32257}\sum_{k=0}^{\infty}
\frac{(1/2)_k(1/6)_k(5/6)_k}{k!(3/4)_k(5/4)_k}
\bigg(\frac{7308019017576087552}{1126532810779723715634459649}\bigg)^{k},\\
&&\xxqdn\sqrt{7}=\frac{32257}{12192}\sum_{k=0}^{\infty}
\frac{(1/2)_k(1/6)_k(5/6)_k}{k!(3/4)_k(5/4)_k}
\bigg(-\frac{1082669486166374401}{166893749263957816334352384}\bigg)^{k}.
 \enm

\section{Series expansions for $\sqrt{11}$}
Taking $p=11$ in \eqref{pell}, we obtain
 \bmn\label{pell-e}
  x^2-11y^2=1.
 \emn
It is not difficult to see that $x_1=10, y_1=3$ is the solution to
\eqref{pell-e}. Thus
 \bnm\qquad
\begin{cases}
x_s=\sum_{k=0}^{\lfloor\frac{s}{2}\rfloor}\binm{s}{2k}11^{k}3^{2k}10^{s-2k},\\
y_s=\sum_{k=0}^{\lfloor\frac{s-1}{2}\rfloor}\binm{s}{1+2k}11^{k}3^{1+2k}10^{s-1-2k}
\end{cases}
\enm
 is also the solution of \eqref{pell-e} according to Lemma
 \ref{lemm-a}. Now we choose $x_3=3970, y_3=1197$ and $x_4=79201, y_4=23880$
 to create 12 series expansions for $\sqrt{11}$.

Substituting p=11, $n=x_3=3970$ and $m=y_3=1197$ into Theorems
\ref{thm-a}-\ref{thm-f}, we get the following six series expansions
for $\sqrt{11}$\,:
 \bnm
&&\xxqdn\xxqdn\xqdn\sqrt{11}=\frac{13167}{3970}\sum_{k=0}^{\infty}\frac{(1/2)_k}{k!}\bigg(\frac{1}{15760900}\bigg)^k,\\
&&\xxqdn\xxqdn\xqdn\sqrt{11}=\frac{3970}{1197}\sum_{k=0}^{\infty}\frac{(1/2)_k}{k!}\bigg(-\frac{1}{15760899}\bigg)^k,\\
&&\xxqdn\xxqdn\xqdn\sqrt{11}=\frac{13167}{3970}\sum_{k=0}^{\infty}
\frac{(1/4)_k(3/4)_k}{k!(3/2)_k}\bigg(\frac{15760899}{62101492202500}\bigg)^{k},\\
&&\xxqdn\xxqdn\xqdn\sqrt{11}=\frac{3970}{1197}\sum_{k=0}^{\infty}
\frac{(1/4)_k(3/4)_k}{k!(3/2)_k}\bigg(-\frac{63043600}{248405937288201}\bigg)^{k},\\
&&\xxqdn\xxqdn\xqdn\sqrt{11}=\frac{13167}{3970}\sum_{k=0}^{\infty}
\frac{(1/2)_k(1/6)_k(5/6)_k}{k!(3/4)_k(5/4)_k}
\bigg(\frac{6706960306781427}{15660406535270116000000}\bigg)^{k},\\
&&\xxqdn\xxqdn\xqdn\sqrt{11}=\frac{3970}{1197}\sum_{k=0}^{\infty}
\frac{(1/2)_k(1/6)_k(5/6)_k}{k!(3/4)_k(5/4)_k}
\bigg(-\frac{62101492202500}{145003736614802587137}\bigg)^{k}.
 \enm

Substituting p=11, $n=x_4=79201$ and $m=y_4=23880$ into Theorems
\ref{thm-a}-\ref{thm-f}, we gain the following six series expansions
for $\sqrt{11}$\,:
 \bnm
&&\xxqdn\sqrt{11}=\frac{262680}{79201}\sum_{k=0}^{\infty}\frac{(1/2)_k}{k!}\bigg(\frac{1}{6272798401}\bigg)^k,\\
&&\xxqdn\sqrt{11}=\frac{79201}{23880}\sum_{k=0}^{\infty}\frac{(1/2)_k}{k!}\bigg(-\frac{1}{6272798400}\bigg)^k,\\
&&\xxqdn\sqrt{11}=\frac{262680}{79201}\sum_{k=0}^{\infty}
\frac{(1/4)_k(3/4)_k}{k!(3/2)_k}\bigg(\frac{25091193600}{39347999779588156801}\bigg)^{k},\\
&&\xxqdn\sqrt{11}=\frac{79201}{23880}\sum_{k=0}^{\infty}
\frac{(1/4)_k(3/4)_k}{k!(3/2)_k}\bigg(-\frac{6272798401}{9836999941760640000}\bigg)^{k},\\
&&\xxqdn\sqrt{11}=\frac{262680}{79201}\sum_{k=0}^{\infty}
\frac{(1/2)_k(1/6)_k(5/6)_k}{k!(3/4)_k(5/4)_k}
\bigg(\frac{265598998427537280000}{246822070099948942419850075201}\bigg)^{k},\\
&&\xxqdn\sqrt{11}=\frac{79201}{23880}\sum_{k=0}^{\infty}
\frac{(1/2)_k(1/6)_k(5/6)_k}{k!(3/4)_k(5/4)_k}
\bigg(-\frac{39347999779588156801}{36566232589911843422208000000}\bigg)^{k}.
 \enm

\section{Series expansions for $\sqrt{13}$}
Fixing $p=13$ in \eqref{pell}, we have
 \bmn\label{pell-f}
  x^2-13y^2=1.
 \emn
We can find that $x_1=649, y_1=180$ is the solution to
\eqref{pell-f} by means of the commands of \emph{Mathematica}:
 \bnm
&&M=180;\\
&&aa[m_{-}]:=\sqrt{13m^2+1};\\
&&For[i=1,i<=M,\\
&&\qquad x=aa[i];\\
&&\qquad If[Mod[x,1]==0,Print[\sqrt{13i^2+1}]]\\
&&\qquad i++\\
&&\qquad ]\\
&&649.
 \enm
 Therefore
 \bnm\qquad
\begin{cases}
x_s=\sum_{k=0}^{\lfloor\frac{s}{2}\rfloor}\binm{s}{2k}13^{k}180^{2k}649^{s-2k},\\
y_s=\sum_{k=0}^{\lfloor\frac{s-1}{2}\rfloor}\binm{s}{1+2k}13^{k}180^{1+2k}649^{s-1-2k}
\end{cases}
\enm
 is also the solution of \eqref{pell-f} in accordance with Lemma
 \ref{lemm-a}. Now we select $x_1=649, y_1=180$ and $x_2=842401, y_2=233640$
 to produce 12 series expansions for $\sqrt{13}$.

Substituting p=13, $n=x_1=649$ and $m=y_1=180$ into Theorems
\ref{thm-a}-\ref{thm-f}, we achieve the following six series
expansions for $\sqrt{13}$\,:
 \bnm
&&\xxqdn\xxqdn\xqdn\sqrt{13}=\frac{2340}{649}\sum_{k=0}^{\infty}\frac{(1/2)_k}{k!}\bigg(\frac{1}{421201}\bigg)^k,\\
&&\xxqdn\xxqdn\xqdn\sqrt{13}=\frac{649}{180}\sum_{k=0}^{\infty}\frac{(1/2)_k}{k!}\bigg(-\frac{1}{421200}\bigg)^k,\\
&&\xxqdn\xxqdn\xqdn\sqrt{13}=\frac{2340}{649}\sum_{k=0}^{\infty}
\frac{(1/4)_k(3/4)_k}{k!(3/2)_k}\bigg(\frac{1684800}{177410282401}\bigg)^{k},\\
&&\xxqdn\xxqdn\xqdn\sqrt{13}=\frac{649}{180}\sum_{k=0}^{\infty}
\frac{(1/4)_k(3/4)_k}{k!(3/2)_k}\bigg(-\frac{421201}{44352360000}\bigg)^{k},\\
&&\xxqdn\xxqdn\xqdn\sqrt{13}=\frac{2340}{649}\sum_{k=0}^{\infty}
\frac{(1/2)_k(1/6)_k(5/6)_k}{k!(3/4)_k(5/4)_k}
\bigg(\frac{1197513720000}{74725388357583601}\bigg)^{k},\\
&&\xxqdn\xxqdn\xqdn\sqrt{13}=\frac{649}{180}\sum_{k=0}^{\infty}
\frac{(1/2)_k(1/6)_k(5/6)_k}{k!(3/4)_k(5/4)_k}
\bigg(-\frac{177410282401}{11070349056000000}\bigg)^{k}.
 \enm

Substituting p=13, $n=x_2=842401$ and $m=y_2=233640$ into Theorems
\ref{thm-a}-\ref{thm-f}, we attain the following six series
expansions for $\sqrt{13}$\,:
\bnm
&&\xxqdn\xxqdn\sqrt{13}=\frac{3037320}{842401}\sum_{k=0}^{\infty}\frac{(1/2)_k}{k!}\bigg(\frac{1}{709639444801}\bigg)^k,\\
&&\xxqdn\xxqdn\sqrt{13}=\frac{842401}{233640}\sum_{k=0}^{\infty}\frac{(1/2)_k}{k!}\bigg(-\frac{1}{709639444800}\bigg)^k,\\
&&\xxqdn\xxqdn\sqrt{13}=\frac{3037320}{842401}\sum_{k=0}^{\infty}
\frac{(1/4)_k(3/4)_k}{k!(3/2)_k}\bigg(\frac{2838557779200}{503588141617471525929601}\bigg)^{k},\\
\enm
 \bnm
 &&\sqrt{13}=\frac{842401}{233640}\sum_{k=0}^{\infty}
\frac{(1/4)_k(3/4)_k}{k!(3/2)_k}\bigg(-\frac{709639444801}{125897035404013061760000}\bigg)^{k},\\
&&\sqrt{13}=\frac{3037320}{842401}\sum_{k=0}^{\infty}
\frac{(1/2)_k(1/6)_k(5/6)_k}{k!(3/4)_k(5/4)_k}
\\&&\qquad\qquad\qquad\quad\times\:
\bigg(\frac{3399219955908352667520000}{357366009225789855782108329051454401}\bigg)^{k},\\
&&\sqrt{13}=\frac{842401}{233640}\sum_{k=0}^{\infty}
\frac{(1/2)_k(1/6)_k(5/6)_k}{k!(3/4)_k(5/4)_k}
\\&&\qquad\qquad\qquad\quad\times\:
\bigg(-\frac{503588141617471525929601}{52943112477670976497371561984000000}\bigg)^{k}.
 \enm

In this paper, we establish numerous series expansions for
$\sqrt{p}$ with $p=2,3,5,7,11,13$. When $p$ are other prime numbers,
series expansions for $\sqrt{p}$, which converge fast, can also be
established in the same way.

\textbf{Acknowledgments}

The work is supported by the Natural Science Foundations of China
(Nos. 11301120, 11201241).


\end{document}